\documentclass[10pt,a4paper, final]{amsart}
\usepackage[english]{babel}
\usepackage{latexsym}

\usepackage[all]{xy}
\usepackage{amssymb}
\usepackage{amscd}

\usepackage{mathrsfs}
\usepackage[cp850]{inputenc}

\theoremstyle{plain}
\newtheorem{theorem}{Theorem}[section]%[chapter]
\newtheorem{proposition}{Proposition}[section]%[chapter]
\newtheorem{cor}{Corollary}[section]%[chapter]
\newtheorem{lemma}{Lemma}[section]%[chapter]
\theoremstyle{definition}

\newtheorem{prob}{Problem}
\numberwithin{equation}{section}
\newcommand{\e}{\varepsilon}
\newcommand{\N}{\mathbb{N}}
\newcommand{\K}{\mathbb{K}}

\newcommand{\adef}{\begin{defn}}
\newcommand{\zdef}{\end{defn}}
\newtheorem{defn}[theorem]{Definition}

\newcommand{\To}{\longrightarrow}

\def\Ext{\operatorname{Ext}}

\def\PB{\operatorname{PB}}
\def\PO{\operatorname{PO}}
\def\dens{\operatorname{dens}}
\def\IM{\operatorname{Im}}

\def\Ext{\operatorname{Ext}}

\def\pd{\operatorname{pd}}
\def\id{\operatorname{id}}
\def\fd{\operatorname{fd}}

\begin{document}

\title{Homological dimensions of Banach spaces}

\author[F\'elix~Cabello~S\'anchez]{F.~Cabello~S\'anchez}
\address{Departamento de Matem\'aticas and IMUEx, Universidad de Extremadura\\
Avenida de Elvas\\ 06071-Badajoz\\ Spain}
\email{fcabello@unex.es}
\author[Jes\'{u}s~M.~F.~Castillo]{J.\,M.\,F.~Castillo}
\address{Departamento de Matem\'aticas\\ Universidad de Extremadura and IMUEx\\
Avenida de Elvas\\ 06071-Badajoz\\ Spain}
 \email{castillo@unex.es}
\author[Ricardo~Garc\'\i{}a]{R.~Garc\'\i{}a}
\address{Departamento de Matem\'aticas\\ Universidad de Extremadura and IMUEx\\
Avenida de Elvas\\ 06071-Badajoz\\ Spain}
\email{rgarcia@unex.es}

%\doi{10.XXXX/smXXXX}
%\date{28/JUL/2020}
\subjclass{46M15, 46M18, 46M10}

\maketitle

\markright{Homological dimensions of Banach spaces}

%\begin{fulltext}

\begin{abstract}
The purpose of this paper is to lay the foundations for the study of the problem of when $\Ext^n(X, Y)=0$ in Banach spaces. We provide a number of examples of couples $X,Y$ so that $\Ext^n(X,Y)$ is (or is not) $0$. We show that $\Ext^n(\mathcal  K, \mathcal  K)\neq 0$ for all $n\in \N$ when $\mathcal K$ is Kadec' space. In particular, both the projective and the injective dimensions of $\mathcal K$ are infinite.

Bibliography: 48~titles.
\end{abstract}

%\begin{keywords}
%Exact sequence, homology, $\Ext^n$ functor, Banach space, quasi Banach space, homological dimension.
%\end{keywords}

\footnotetext{This research has been supported in part by MINCIN, Project MTM2016-76958-C2-1-P, and Junta de Extremadura, Project IB16056.}

\section{Introduction}

The ``homological theory of Banach spaces'', as it has been developed so far,  wheels around the existence, meaning and relationships between the functors $\mathfrak L$ (linear continuous operators) and  $\Ext$ (exact sequences of Banach spaces modulo equivalence). Such relations are based on two facts:
\begin{itemize}
\item $\Ext$ is the derived functor of $\mathfrak L$.
\item There is an object, the long homology sequence, that connects both.
\end{itemize}
The expositions \cite{mitc,fresi} can serve as a basic introduction to Yoneda Ext functors in arbitrary exact categories.

The purpose of this paper is to lay the foundations for the study of $\Ext^n$ in the category of Banach spaces. That purpose sets the general tone of the paper: the first definition the reader will encounter is that of exact sequence of length $n$ and  $\Ext^n$, the $n$-{th} derived  functor of $\mathfrak L$ (the functor Hom in our ambient category of Banach spaces). We will not dig in this paper on the precise way in which the derivation of functors works; rather, we will take the long homology sequences as the cornerstone object that operatively defines derivation, as can be seen in Section~\ref{sec:red}. We have included an Appendix with a succint description of the homology sequences and the material on pushout and pullback sequences that is indispensable to understand the paper.

Section~\ref{sec:extn} contains the main results of the paper. In general these combine the apparatus of homological algebra with specific results on Banach spaces: some of them clearly belong to the Banach space lore, while others are very recent.

Section~\ref{sec:dim} contains some material on the projective and injective dimension of Banach spaces. We follow ideas of Wodzicki and we develop some of his results sketched in \cite{wod},  which is still the main reference on this topic. Our main result in this line is that $\Ext^n(\mathcal K, \mathcal K)$ is nonzero when $\mathcal K$ is Kadec space and thus both the projective and injective dimensions of $\mathcal K$ are infinite.

Section~\ref{sec:Q} contains some observations on the ``homological interaction'' between the category of Banach and the larger one of quasi Banach spaces.
\medskip

Finally, let us remark that, while the study of $\Ext^n$ in Banach spaces is still incipient, the connections between homological algebra and the theory of locally convex
spaces was firmly established by Palamodov \cite{palaf,palah} very early.
We refer the reader to Wengenroth's monograph \cite{w-LNM} for a nice introduction to this topic and to \cite{weng0} for more advanced results.\\

{\bf Acknowledgement.} The authors acknowledge the tremendous efforts of the referee in elaborating a thorough report on a text containing an intolerably large number of mistakes, typos and inaccuracies. His/her notes helped us during the preparation of the present readable version of the paper.

\section{The functor $\Ext^n$ in Banach spaces}\label{sec:def-extn}
We introduce some notation and quote the necessary results from Yoneda extension theory.  Most of these are given in full detail in \cite[\textsection 6]{fresi} and  \cite[VII]{mitc}.

\subsection{Exact sequences} An exact sequence of Banach spaces is a (finite or infinite) diagram
$$\begin{CD} \cdots@>>> Y_{i-1} @>>> Y_i @>>> Y_{i+1}@>>>  \cdots\qquad(\mathscr Y)\end{CD}$$
with $i\in {\mathbb Z}$, formed by Banach spaces and (linear continuous) operators such that the kernel of each arrow coincides
with the image of the preceding one. An $n$-exact sequence $\mathscr E$ between $Y$ and $X$ is an exact sequence
$$\begin{CD}
0@>>> Y@>>> E_1 @>>> \cdots @>>> E_n @>>> X @>>> 0\qquad(\mathscr E)
\end{CD}$$
having $n$ terms between $Y$ and $X$ and all the rest $0$. We then call $n$ the length of $\mathscr E$; $1$-exact sequences
are the popular short exact sequences. A morphism $\phi: \mathscr E \longrightarrow \mathscr F$ between two $n$-exact sequences is a commutative diagram
$$\begin{CD}
0@>>>Y @>>> E_1 @>>>  \cdots @>>> E_n @>>> X@>>> 0 \qquad(\mathscr E)\\
&&@VV\phi_{-} V @VV\phi_{1} V  & & @VV\phi_{n} V @VV\phi_{+} V\\
0@>>> Y' @>>> F_1 @>>> \cdots @>>> F_n @>>> X'@>>> 0 \qquad(\mathscr F)
\end{CD}$$
Given two $n$-exact sequences with the same end spaces $Y,X$, we write $\mathscr E \longrightarrow \mathscr F$ (or $\mathscr F \longleftarrow \mathscr E$) to indicate the existence of a morphism
$\phi:\mathscr E \longrightarrow \mathscr F$ with $\phi_-={\bf I}_Y$ and $\phi_+={\bf I}_X$.
We introduce an equivalence relation on the class of $n$-exact sequences with fixed ends $Y$ and $X$ by declaring $\mathscr E \sim \mathscr F$ if and only if there are finitely many $n$-exact sequences with the same ends  $(\mathscr G_{j})_{1\leq j \leq k}$ so that
\begin{equation}\label{eq:E=F}
\mathscr E \longrightarrow  \mathscr G_1 \longleftarrow \mathscr G_2 \longrightarrow \cdots \longleftarrow \mathscr G_k \longrightarrow \mathscr F.
\end{equation}
It can be shown (see \cite[6.40]{fresi}) that it only takes two exact sequences and three morphisms to establish $\mathscr E \sim \mathscr F$, that is, thus
$\mathscr E \sim \mathscr F$ if and only if there is a chain as (\ref{eq:E=F}) with $k=2$.

We define $\Ext^n(X, Y)$ to be the \emph{set} of equivalence classes of $n$-exact sequences between $Y$ and $X$. Let us agree that $\Ext(X,Y)=\Ext^1(X,Y)$; if $Y=X$ we just write  $\Ext^n(X)$.
We emphasize that $\Ext^n(X, Y)$ is a set, even is the \emph{class} of all $n$-exact sequences between $Y$ and $X$ is ``too large'' to be a set.
 See \cite[6.20]{fresi} for these theoretical issues. We write $\mathscr E\in\Ext^n(X,Y)$ when $\mathscr E$ is an $n$-exact sequence between $Y$ and $X$, with the understanding that it is the equivalence class of $\mathscr E$ what really belongs to $\Ext^n(X,Y)$.

\subsection{Splicing and cutting sequences}
The set  $\Ext^n(X, Y)$ admits a natural linear structure whose operations are defined by means of pullbacks and pushouts. Our results are so pedestrian that everything we need to know about that structure is that it exists (and so it makes sense to say that the homology sequences are exact; see the Appendix B for this topic) and how to identify the zero element, which is a little different, depending on whether $n=1$ or $n\geq 2$:
The zero in $\Ext(X,Y)$ is the (class of the) direct sum sequence
$$\begin{CD}
0@>>> Y@>\imath>> Y \oplus X@>\pi>> X  @>>> 0
\end{CD}$$
where $\imath(y)=(y,0)$ and $\pi(y,x)=x$. If $n\geq 2$, the zero of $\Ext^n(X,Y)$ is the (class of the) sequence
$$\begin{CD}
0@>>> Y@= Y @>0>> \cdots @>0>> X @= X @>>> 0.
\end{CD}$$
We write $\Ext^n(X,Y)=0$ if every $n$-exact sequence between $Y$ and $X$ is equivalent to the zero sequence. Two elements $\mathscr E\in \Ext^n(Z,Y)$ and $\mathscr F\in \Ext^m(X,Z)$ can be spliced through $Z$ to get the $(n+m)$-exact sequence
$$\xymatrixrowsep{0.5pc} \xymatrixcolsep{1.6pc}
\xymatrix{0\ar[r] & Y\ar[r]^{} &E_1\ar[r]  &\cdots \ar[r]  &E_n \ar[rr]  \ar[dr]  &&F_1 \ar[r] &\cdots \ar[r] &F_n\ar[r]  &X\ar[r]& 0\\
& & && & Z \ar[ur]    }
$$
denoted by $\mathscr E\mathscr F$. It is easy to see that the class of $\mathscr E\mathscr F$ in $\Ext^{m+n}(X,Y)$ depends only on the classes of $\mathscr E$ and $\mathscr F$ and that if $\mathscr E\sim 0$ or $\mathscr F\sim 0$, then $\mathscr E\mathscr F\sim 0$.
And conversely, if $n\geq 2$, then every exact sequence
$$
\xymatrixcolsep{2.7pc}
\xymatrix{0\ar[r] & Y\ar[r]^{f_0} &F_1\ar[r]^{f_1}  &\cdots\ar[r]  &F_n \ar[r]^{f_n}  &X\ar[r] & 0 & (\mathscr F)
}
$$
can be cut into shorter pieces as follows: choose $1<i\leq n$; as $\mathscr F$ is exact at $F_i$ we have $\ker f_i=\IM f_{i-1}$ and, if we call $Z$ this space, we have two exact sequences
$$
\xymatrixcolsep{2.6pc}
\xymatrix{0\ar[r] & Y\ar[r]^{f_0} &F_1\ar[r]^{f_1}  &\cdots\ar[r]  &F_{i-1} \ar[r]^{f_{i-1}}  &Z\ar[r] & 0 & (\mathscr L)\\
0\ar[r] & Z\ar[r]^{\text{inclusion}} &F_i\ar[r]^{f_i}  &\cdots\ar[r]  &F_{n} \ar[r]^{f_i}  &X\ar[r] & 0 & (\mathscr R)
}
$$
and, clearly, $\mathscr F= \mathscr L \mathscr R$.
An obvious consequence is:

\begin{cor}\label{n-cor}  {\rm (a)} If $\Ext^n(X, \cdot)=0$, then $\Ext^m(X, \cdot)=0$ for all $m>n$.

 {\rm (b)} If $\Ext^n(\cdot,Y)=0$, then $\Ext^m(\cdot,Y)=0$ for all $m>n$.
\end{cor}

As a rule  $\mathscr{EF}\sim 0$ does not imply that $\mathscr{E}\sim 0$ or $\mathscr{F}\sim 0$; see \cite[Section 6.4]{cck} or \cite[Section 5.3]{ext2} for some striking examples in which  $\mathscr{E} = \mathscr{F}$. One has, however:
\begin{lemma}\label{lem:EF=0=>F=0} {\rm (a)}
Let $\mathscr E$ be a short exact sequence $\xymatrixcolsep{1.75pc}
\xymatrix{
0 \ar[r] & Y \ar[r] & E \ar[r] & Z\ar[r] & 0
}
$ and $\mathscr F\in \Ext^n(X,Z)$, where $n\geq 1$. If $\Ext^n(X,E)=0$ and $\mathscr{EF}\sim 0$ in $\Ext^{n+1}(X,Y)$, then  $\mathscr F\sim 0$.

{\rm (b)}
Let $\mathscr F$ be a short exact sequence $
\xymatrix{
0 \ar[r] & Z \ar[r] & F \ar[r] & X\ar[r] & 0
}
$ and $\mathscr E\in \Ext^n(Z,Y)$, where $n\geq 1$. If $\Ext^n(F,Y)=0$ and $\mathscr{EF}\sim 0$ in $\Ext^{n+1}(X,Y)$, then  $\mathscr E\sim 0$.
\end{lemma}

\begin{proof}
(a) Let us take a look at the following section of the ``covariant'' homology sequence associated to $X$ and $\mathscr E$:
$$
\xymatrixcolsep{1.2pc}
\xymatrix{
\dots \ar[r] & \Ext^{n}(X,Y)  \ar[r]^0 & \Ext^{n}(X,E)  \ar[r]^0 & \Ext^{n}(X,Z)\ar[rr]^-{\text{injective}} & & \Ext^{n+1}(X,Y)\ar[r] &\dots
}
$$
and note that $\Ext^{n}(X,Z)\To \Ext^{n+1}(X,Y)$ acts sending $\mathscr F$ to $\mathscr{EF}$. To prove (b) use the contravariant sequence associated to $Y$ and $\mathscr F$.\hfill $_\square$
\end{proof}

\section{Reduction of length}\label{sec:red}
When working in categories with enough projective or injective elements, as it is the case of Banach spaces, there is a well-known representation of $\Ext$ in terms of operators.
Let us begin with the projective case.

Let $X$ be a Banach space. A projective presentation of $X$ is a short exact sequence
\begin{equation}\label{eq:kX->P->X}\tag{$\mathscr P$}
\xymatrix{
0\ar[r] & \kappa(X) \ar[r] & \mathcal P  \ar[r]^\pi & X \ar[r] & 0
}
\end{equation}
where $\mathcal{P}$ is a projective Banach space (necessarily isomorphic to $\ell_1(I)$ for some index set $I$ by a result of K\"othe \cite[3(6)]{koethe}). It is easy to see that  $\mathcal{P}$ is projective if and only if $\Ext(\mathcal{P},\cdot)=0$, in which case $\Ext^{n}(\mathcal{P},\cdot)=0$ for all $n \geq 1$, by Corollary~\label{n-cor}(a).
Now, if $Y$ is another Banach space we can run the (contravariant) homology sequence to obtain a long exact sequence
$$
\xymatrixrowsep{1pc}\xymatrixcolsep{1.65pc}
\xymatrix{
0 \ar[r]&\frak L(X, Y) \ar[r] & \frak L(\mathcal P, Y) \ar[r] & \frak L(\kappa(X), Y) \ar[r] &\Ext(X,Y)\\
& \ar[r] &\Ext(\mathcal P,Y) \ar[r] &\Ext(\kappa(X),Y) \ar[r] &\Ext^2(X,Y) \\
 \ar[r]&\dots \ar[r] & \Ext^{n}(\mathcal P, Y) \ar[r] & \Ext^{n}(\kappa(X), Y) \ar[r] &\Ext^{n+1}(X,Y)\ar[r] &\dots\\
}
$$
As $\Ext^{n}(\mathcal P, Y)=0$ for all $n\geq 1$.
This shows that
\begin{equation}\label{eq:reduction}
\Ext^{n+1}(X,Y)= \Ext^{n}(\kappa(X), Y)
\end{equation}
in the sense that, up to equivalence, every $(n+1)$-exact sequence between $Y$ and $X$ can be obtained as the splicing $\mathscr E\mathscr P$:
$$
\xymatrixrowsep{0.5pc}\xymatrixcolsep{2.25pc}
\xymatrix{
0\ar[r] & Y\ar[r] & E_{1}\ar[r] &\dots  \ar[r] & E_{n} \ar[rr] \ar[dr] & &\mathcal P \ar[r] &X \ar[r] & 0\\
&&&&&\kappa(X)\ar[ur]
}
$$
Besides, $\mathscr E\sim 0$ in $\Ext^{n}(\kappa(X), Y)$
if and only if $\mathscr E\mathscr P\sim 0$ in $\Ext^{n+1}(X, Y)$. This is just a particular case of Lemma~\ref{lem:EF=0=>F=0}(b).
The interpretation for $\Ext(X,Y)$ is somehow different. Since $\frak L(\mathcal P, Y)\neq 0$  the exactness at  $\Ext(X,Y)$ means that every short exact sequence $0\To Y\To E\To X\To 0$ fits into a (necessarily pushout) diagram
$$
\xymatrix{
0\ar[r] & \kappa(X) \ar[d]^u\ar[r] & \mathcal P \ar[r]\ar[d] & X\ar[r]\ar@{=}[d] & 0\\
0\ar[r] & Y \ar[r] & E \ar[r] & X\ar[r] & 0
}
$$
for some $u\in \frak L(\kappa(X), Y)$ which admits an extension to $\mathcal{P}$ if and only if the extension splits. This obviously follows from the lifting property of $\mathcal P$. Thus,
\begin{equation}\label{eq:Ext=L/restrictions}
\Ext(X,Y)=\frac{\frak L(\kappa(X), Y)}{\imath^*[\frak L(\mathcal{P}, Y)]},
\end{equation}
where $\imath^*: \frak L(\mathcal{P}, Y)\To \frak L(\kappa(X), Y)$ is the restriction map.

Let us assume that for each Banach space $X$ a projective presentation as in (\ref{eq:kX->P->X})  has been chosen. Then we can attach to every $X$ a sequence of  \,``kernels'' inductively defined as follows: $\kappa^{1}(X)=\kappa(X)$ and $\kappa^{n+1}(X)=\kappa(\kappa^{n}(X))$. For example, we can take $\kappa(X)=\ker Q$ where $Q:\ell_1(B_X) \To X$ is the natural quotient map, but we prefer not to be so specific.

For each $n\geq 1$, by successive splicing, we can construct an $n$-exact sequence
$$
\xymatrixcolsep{1.2pc}
\xymatrixrowsep{0.5pc}
\xymatrix{
0 \ar[r] & \kappa^{n}X  \ar[r]   & \mathcal P_n \ar[dr]   \ar[rr]  & & \cdots  \ar[dr] \ar[rr]  & & \mathcal P_2 \ar[rr]  \ar[dr]  & & \mathcal P  \ar[r]   &  X \ar[r] & 0\\
& & &\kappa^{n-1}X \ar[ur] && \kappa^{2}X \ar[ur] && \kappa^{1}X \ar[ur]
}
$$
Now, given $\mathscr E\in\Ext^n(X,Y)$, by decomposition into short sequences and applying successively the lifting property of $\mathcal P, \mathcal P_2\dots$ one obtains a commutative diagram
$$
\xymatrixrowsep{0.5pc}
\xymatrixcolsep{1.17pc}
\xymatrix{
0 \ar[r] & \kappa^{n}X \ar[r]\ar[dddd]^u   & \mathcal P_n  \ar[dddd]\ar[dr]   \ar[rr]  & & \cdots  \ar[dr] \ar[rr]  & & \mathcal P_2 \ar[dddd] \ar[rr]  \ar[dr]  & & \mathcal P   \ar[dddd] \ar[r]   &  X \ar[r] \ar@{=}[dddd] & 0\\
& & &\kappa^{n-1}X \ar[ur] \ar[dddd] && \kappa^{2}X \ar[dddd] \ar[ur] && \kappa^{1}X \ar[dddd] \ar[ur]
\\
\\
\\
0 \ar[r] & Y \ar[r]   & E_n \ar[dr]   \ar[rr]  & & \cdots  \ar[dr] \ar[rr]  & &  E_2 \ar[rr]  \ar[dr]  & &  E_1  \ar[r]   &  X \ar[r] & 0\\
& & &K_{n-1} \ar[ur] && K_2 \ar[ur] && K_1 \ar[ur]
}
$$

We have:

\begin{proposition}\label{reduction} Let $X$ and $Y$ Banach spaces. Then $$\Ext^n(X, Y)= \Ext^{n-1}(\kappa(X), Y) = \cdots = \Ext^1(\kappa^{n-1}(X), Y)
= \frac{\mathfrak L(\kappa^{n}(X), Y)}{\imath^*[\frak L(\mathcal{P}_n, Y)]}.$$

\end{proposition}

\begin{proof}
All the identities, but the last one, are particular cases of  (\ref{eq:reduction}), taking into account that $\kappa^{k+1}(X)=\kappa(\kappa^k(X))$. The last one is just (\ref{eq:Ext=L/restrictions}) applied to  $\kappa^{n-1}(X)$.
 \hfill $_\square$
\end{proof}

Proceeding by categorical duality (reversing the arrows) we can do the injective version as well. First, an injective presentation of a Banach space $Y$ is a short exact sequence
\begin{equation}\label{eq:Y->I->ckY}\tag{$\mathscr I$}
\xymatrix{
0 \ar[r] & Y\ar[r] &\mathcal I \ar[r]^-\pi & \mathcal I/\imath[Y]=c\kappa(Y)\ar[r] & 0
}
\end{equation}
where $\mathcal I$ is an injective Banach space (necessarily a complemented subspace of some $\ell_\infty(I)$ and, therefore, an $\mathscr L_\infty$-space \cite[Corollary on p. 335]{lindrosep}). Note that $\mathcal{I}$ is injective if and only if $\Ext^{n}(\cdot, \mathcal{I})=0$ for all $n \geq 1$. If $X$ is another Banach space we can activate the covariant homology sequence to obtain the long exact sequence
$$
\xymatrixrowsep{1pc}\xymatrixcolsep{1.6pc}
\xymatrix{
0 \ar[r]&\frak L(X, Y) \ar[r] & \frak L(X, \mathcal I) \ar[r] & \frak L(X, c\kappa(Y)) \ar[r] &\Ext(X,Y)\\
& \ar[r] &\Ext(X, \mathcal I) \ar[r] &\Ext(X, c\kappa(Y)) \ar[r] &\Ext^2(X,Y) \\
 \ar[r]&\dots \ar[r] & \Ext^{n}(X, \mathcal I) \ar[r] & \Ext^{n}(X, c\kappa(Y)) \ar[r] &\Ext^{n+1}(X,Y)\ar[r] &\dots\\
}
$$
As $\Ext^{n}(X, \mathcal I)=0$ for all $n\geq 1$, we have
\begin{equation}\label{eq:reduction-inj}
\Ext^{n+1}(X,Y)= \Ext^{n}(X, c\kappa(Y));\qquad \Ext^1(X,Y)=\frac{\frak L(X, c\kappa(Y))}{\pi_*[\frak L(X, \mathcal I)]},
\end{equation}
where $\pi_*[\frak L(X, \mathcal I)]$ consists of those operators $X\To c\kappa(Y)$ that can be lifted to $\mathcal{I}$.

Now, if we fix an injective presentation as in (\ref{eq:Y->I->ckY})
 ``for each Banach space'' $Y$ and we define recursively $c\kappa^1(Y)=c\kappa(Y)$ and $c\kappa^{k+1}(Y)=c\kappa(c\kappa^{k}(Y))$.  For example, one could take $c\kappa(Y)=\text{coker}(J)$ where $J:Y \To \ell_\infty(B_{Y^*})$ is the obvious embedding, but some flexibility is convenient here.
 Then, for each $n\geq 1$, we can construct an $n$-exact sequence
 $$
\xymatrixcolsep{1.18pc}
\xymatrixrowsep{0.5pc}
\xymatrix{
0 \ar[r] & Y  \ar[r]   & \mathcal I \ar[dr]   \ar[rr]  & & \mathcal I_2  \ar[dr] \ar[rr]  & & \cdots \ar[rr]  \ar[dr]  & & \mathcal I_n  \ar[r]^\pi   &  c\kappa^{n}Y \ar[r] & 0\\
& & & c\kappa Y \ar[ur] &&  c\kappa^{2}Y \ar[ur] &&  c\kappa^{n-1}Y \ar[ur]
}
$$
and we have the injective counterpart of  Proposition~\ref{reduction}:

\begin{proposition}\label{reduction-inj} Let $X$ and $Y$ Banach spaces. Then, for every $n\geq 1$,
$$\Ext^n(X, Y)= \Ext^{n-1}(X , c\kappa(Y)) = \cdots = \Ext^1(X, c\kappa^{n-1}(Y))
= \frac{\mathfrak L(X, c\kappa^{n}(Y) )}{{\pi_*[\frak L(X, \mathcal I_n)]}}.$$
\end{proposition}

\section{$\Ext^n$ problems on Banach spaces}\label{sec:extn}
So far, the study of $\Ext^n$ for Banach spaces has been focused almost exclusively on
short exact sequences. These are somewhat exceptional for two reasons.
First, in a short exact sequence
\begin{equation}\tag{$\mathscr E$}
\xymatrix{0\ar[r] & Y\ar[r]^\imath & E \ar[r]^\pi & X\ar[r] & 0
}
\end{equation}
the map $\imath$ is an isomorphic embedding and $\pi$ defines an isomorphism between  $E/\imath[Y]$ and $X$. For this reason the middle space $E$ is often called a ``twisted sum'' of $Y$ and $X$ in this setting.
Second, and more important, the equivalence relation in $\Ext^1$ simpler than in the case of longer sequences. Indeed
any operator $u$ fitting in a commutative diagram
$$\begin{CD}
0@>>> Y @>>> E @>>> X @>>> 0 \\
&&@| @VVV @|\\
0@>>> Y @>>> F @>>> X@>>> 0\end{CD} $$
with exact rows is an isomorphism by the well-known $3$-lemma and the open mapping theorem. In particular, $\mathscr E\sim 0$ in $\Ext(X,Y)$, that is, it is equivalent to the direct sum sequence $\xymatrix{0\ar[r] & Y\ar[r] & Y\oplus X \ar[r] & X\ar[r] & 0}$ if and only if it splits, that is, there is $P\in\frak L(E,Y)$ such that $P\imath={\bf I}_{Y}$ or, equivalently, there is $S\in\frak L(X, E)$ such that $\pi S={\bf I}_{X}$.

Several important Banach characterizations adopt the form $\Ext(X,Y)=0$. Some easy examples are:
\begin{itemize}
\item $X$ is projective $\iff$ $\Ext(X, \cdot)=0$.
\item $Y$ is injective $\iff$ $\Ext(\cdot, Y)=0$.
\item $Y$ is $\aleph$-injective $\iff$ $\Ext(X, Y)=0$ for every space $X$ with density character strictly less than $\aleph$; \cite[Proposition 5.3]{2132}. When $\aleph=\aleph_1$ this property is referred to as separable injectivity.
\end{itemize}
More sophisticated results to be mentioned are:
\begin{itemize}
\item $X$ is an $\mathscr L_1$-space $\iff$ $\Ext(X, U)=0$ for every
 for every Banach space $U$ complemented in its bidual (or just reflexive). The implication $\implies$ is a particular case of Lindenstrauss' lifting (namely the Lemma in \cite{lind}; see also \cite[Proposition 2.1]{kp}). The converse can be seen in \cite[Proposition 2]{cabecastuni}.

\item $Y$ is an $\mathscr L_{\infty}$-space $\iff$ for every sequence of finite dimensional Banach spaces $(F_n)$ one has  $\Ext(\ell_1(F_n), Y)=0$. This is clearly equivalent to \cite[Proposition 3.1]{cms-LP1}.

\item A separable Banach space $Y$ is isomorphic to $c_0$ $\iff$ $\Ext(X,Y) = 0$ for every separable Banach space $X$. The implication $\implies$ is Sobczyk's Theorem \cite{sobc}. The converse is due to Zippin \cite{zipp}.
\item The Johnson-Zippin theorem \cite[Corollary 3.1]{jz*} asserts that $\Ext(H^*, Y)=0$ for every subspace $H$ of $c_0$ and every $\mathscr L_\infty$-space $Y$.
\end{itemize}
Many $3$-space problems (see \cite{castgonz} for general information on $3$-space problems) reduce to know whether or not $\Ext(X,Y)$ vanishes for suitable choices of $X,Y$. In particular, the so-called Palais problem: Is $\Ext(\ell_2)=0$? negatively solved by Enflo, Lindenstrauss and Pisier \cite{elp} and then by Kalton and Peck \cite{kaltpeck}.
Problems of the type $\Ext^n=0$ have scarcely, if ever, been considered in Banach space theory. Accordingly, before entering into more serious matters we establish the $n$-versions of the previous results:

\begin{theorem}\label{negative} For each of the following choices of $X$ and $Y$ one
has  $\Ext^n(X, Y)=0$ for all $n \geq 1$.
\begin{enumerate}
\item $Y$ is injective or $X$ is projective.
\item $X$ is separable and $Y$ is separably injective. More generally, if $\dens(X)<\aleph$ and $Y$ is $\aleph$-injective. In particular, this yields
\item Sobczyk's theorem or order $n$: $\Ext^n(X, c_0) = 0$ for every separable space $X$.
\item Lindenstrauss lifting of order $n$: $X$ is an $\mathscr L_1$-space and $Y$ is complemented in its bidual. In particular, for every measure $\mu$ one has $\Ext^n(L_1(\mu)) = 0$.
\item Johnson-Zippin theorem of order $n$: if $H$ is a subspace of $c_0$ and $Y$ is an $\mathscr L_\infty$-space, then $\Ext^n(H^*, Y)=0$.
\end{enumerate}
\end{theorem}

\begin{proof} (1) is obvious. (2) is almost equally obvious after one realizes that if $\dens(X)<\aleph$, then one can choose a projective presentation ($\mathscr P$) with $\mathcal P=\ell_1(I)$ and $|I|<\aleph$, so that $\dens (\kappa(X))\leq \dens(\ell_1(I))=|I|<\aleph$. Iterating the argument we see that one can choose $\mathcal P_n$, and thus
$\kappa^n(X)$, of density character less than $\aleph$. Now, by Proposition~\ref{reduction} we have
$$
\Ext^n(X, Y)
= \frac{\mathfrak L(\kappa^{n}(X), Y)}{\imath^*[\frak L(\mathcal{P}_n, Y)]}
$$
and the quotient space is zero since, by the very definition of $\aleph$-injectivity, every operator
$\kappa^{n}(X)\To Y$ extends to $\mathcal P_n$.
To prove (4), we use Proposition \ref{reduction} in the form $\Ext^n(X, Y)= \Ext^1(\kappa^{n-1}(X), Y)$.
A classical result \cite[Proposition 5.2]{lindrosep} yields that if $X$ is an $\mathscr L_1$-space then  $\kappa(X)$ is again an $\mathscr L_1$-space, and thus Lindenstrauss lifting
is enough to conclude. Finally, (5) reduces to the basic case $n=1$ by means of Proposition~\ref{reduction-inj}: we have $\Ext^n(X, Y) = \Ext(X, c\kappa^{n-1}(Y))$
and again \cite[Proposition 5.2]{lindrosep} tell us that $c\kappa^{n-1}(Y)$ must be an $\mathscr L_\infty$-space.
\hfill $_\square$
\end{proof}

The previous results provide a few partial answers to the following general problems:

\begin{itemize}
\item Characterize the Banach spaces $X, Y$ for which
$\Ext^n(X, Y)=0$.
\item Characterize the Banach spaces $X$ for which
$\Ext^n(X)=0$.
\end{itemize}

Nevertheless, it would be a mistake to think that order $n$ results are a simple generalization of order 1 results. The proof of the following result is based on Bourgain's construction of an uncomplemented subspace of $\ell_1$ isomorphic to $\ell_1$. As far as we know, the significance of this fact in the study of $\Ext^2$ was first noticed by Wodzicki  \cite{wod} to whom parts (3) and (4) of the following result are due.

\begin{proposition}\label{bour}
There exist a Banach space $\mathcal B$ such that:
\begin{enumerate}
\item $\mathcal B$ is not an $\mathscr L_\infty$-space (equivalently, $\mathcal B^*$ is not an $\mathscr L_1$-space).
\item $\Ext^n(S, \mathcal B)=0$ for every separable space $S$ and all $n\geq 2$, but $\Ext^2(\cdot, \mathcal B)\neq0$.
\item $\Ext^n(\mathcal B^*,\cdot)=0$ for all $n\geq 2$.
\item $\Ext^n(\cdot,\mathcal B^{**})=0$ for all $n\geq 2$.
\end{enumerate}
\end{proposition}
\begin{proof}
In proving \cite[Theorem 7]{bour} Bourgain shows that there is some constant $C>0$ so that for every $\varepsilon>0$ and every sufficiently large $n\in \N$  there is $N(n)$ an $n$ dimensional subspace $E_n$ of $\ell_1^{N(n)}$ which is $C$-isomorphic to $\ell_1^n$ and so that every projection $P:\ell_1^{N(n)}\To E_n$ has $\|P\|\geq C^{-1} (\log \log n)^{1-\varepsilon}$.
Considering the exact sequence
\begin{equation}\label{seq:Enl1}
\xymatrix{
0\ar[r] & E_n \ar[r] & \ell_1^{N(n)} \ar[r] & \ell_1^{N(n)}/E_n \ar[r] & 0
}
\end{equation}
and the adjoint sequence
\begin{equation}\label{seq:En*}
\xymatrix{
0\ar[r] & \big( \ell_1^{N(n)}/E_n\big)^*= E_n^\perp \ar[r] & \ell_\infty^{N(n)} \ar[r]^{Q_n} & E_n^* \ar[r] & 0
}
\end{equation}
we see that any linear section $S_n$ of the quotient map in the later sequence has norm at least
$C^{-1} (\log \log n)^{1-\varepsilon}$
since $S_n^*$ is a projection of $\ell_1^{N(n)}$ onto $E_n$. Besides each $E_n^*$ is $C$-isomorphic to $\ell_\infty^n$. Amalgamating the sequences (\ref{seq:En*}) we obtain a short exact sequence
\begin{equation}\label{seq:coEn*}
\xymatrix{
0\ar[r] & c_0( E_n^\perp) \ar[r] & c_0(\ell_\infty^{N(n)}) \ar[r]^Q & c_0(E_n^*) \ar[r] & 0
}
\end{equation}
which does not split since if $S$ is a linear section of $Q$, then the restriction of $S$ to the $n$-th coordinate followed by the obvious projection of  $c_0(\ell_\infty^{N(n)})$
onto the $n$-th factor is a section of $Q_n$. The same argument applies to
\begin{equation}\label{seq:looEn*}
\xymatrix{
0\ar[r] & \ell_\infty( E_n^\perp) \ar[r] & \ell_\infty(\ell_\infty^{N(n)}) \ar[r]^Q & \ell_\infty(E_n^*) \ar[r] & 0.
}
\end{equation}
If we denote $c_0( E_n^\perp)$ by $\mathcal B$ (for Bourgain), then  $c_0(\ell_\infty^{N(n)})$ is isometric to $c_0$, while $c_0(E_n^*)$ is isomorphic to $c_0$ and (\ref{seq:coEn*}) provides a nontrivial exact sequence of the form $0\To \mathcal B\To c_0\To c_0\To 0$ ---which is kind of a ``separably injective resolution of length 1'' for $\mathcal B$. Since the bidual of $\mathcal B$ is naturally isometric to $\ell_\infty( E_n^\perp)$ the nontriviality of (\ref{seq:looEn*}) implies that $\mathcal B$ cannot be an $\mathscr L_\infty$-space since the bidual of any $\mathscr L_\infty$-space is an injective Banach space. This proves (1).

To prove the first part of (2), take a separable Banach space $X$. Running the covariant sequence with the first variable fixed at $X$  we obtain the exact sequence
$$
\xymatrixrowsep{1pc}
\xymatrix{
\dots \ar[r] & \frak L(X, c_0) \ar[r] & \Ext(X,\mathcal B) \ar[r] &\Ext(X, c_0)\\
 \ar[r] &\Ext(X, c_0) \ar[r] &\Ext(X, \mathcal B) \ar[r] &\Ext^2(X,c_0) \\
 \dots \ar[r] & \Ext^{n}(X, c_0) \ar[r] & \Ext^{n}(X, \mathcal B) \ar[r] &\Ext^{n+1}(X,c_0)\ar[r] &\dots\\
}
$$
As $\Ext^{n}(X, c_0)=0$ for all $n\geq 1$ we see that  $\Ext^{n}(X, \mathcal B)=0$ for all $n\geq 2$. Replacing $\mathcal B$ by $\mathcal B^{**}$ and $c_0$ by $\ell_\infty$ and leaving the separability assumption on $X$ one obtains (4), although in this case we can stop at $\Ext^2$ in view of Corollary~\ref{n-cor}. To prove (3) just use the contravariant sequence and the projective presentation (actually resolution)
$0\To\ell_1\To\ell_1\To\mathcal B^*\To 0$ adjoint to (\ref{seq:coEn*}).

Finally, to prove that $\Ext^2(\cdot, \mathcal B)\neq 0$, consider the injective presentation of $\mathcal B$
\begin{equation}\tag{$\mathscr I$}
\xymatrix{
0\ar[r] & \mathcal B \ar[r] &\ell_\infty \ar[r] &  \ell_\infty/\mathcal B=c\kappa(\mathcal B) \ar[r] & 0
}
\end{equation}
provided by the embedding
$\mathcal B= c_0( E_n^\perp) \To c_0(\ell_\infty^{N(n)}) \To \ell_\infty(\ell_\infty^{N(n)})=\ell_\infty$.
Now take an injective presentation of $c\kappa(\mathcal B)$
\begin{equation}\tag{$\mathscr H$}
\xymatrix{
0\ar[r] & c\kappa(\mathcal B)\ar[r]^\jmath &\mathcal H \ar[r] & c\kappa^2(\mathcal B) \ar[r] & 0
}
\end{equation}
and splice them to get the 2-exact sequence
\begin{equation}\tag{$\mathscr I\!\mathscr H$}
\xymatrixrowsep{0.5pc}
\xymatrix{0\ar[r] & \mathcal B\ar[r] &  \ell_{\infty}  \ar[rr] \ar[dr]  &&\mathcal H  \ar[r] &c\kappa^2(\mathcal B) \ar[r]& 0 \\
& && c\kappa(\mathcal B) \ar[ur]   }
\end{equation}
As $\ell_\infty$ is injective we know from Lemma~\ref{lem:EF=0=>F=0}(a) that if $\mathscr I\mathscr H\sim 0$, then $ \mathscr H$ splits which cannot be. Indeed, if $\mathscr H$ splits the ``subspace''  $c\kappa(\mathcal B)=\ell_\infty/\mathcal B$ would be injective, as a complemented subspace of an injective space. To see that this is not the case we first observe that $\mathcal B=c_0( E_n^\perp)$ contains a complemented subspace isomorphic to $c_0$ (just pick a vector in each  $ E_n^\perp$), so that $\mathcal B= \mathcal C\oplus \mathcal D$, with $\mathcal C$ isomorphic to $c_0$.
On the other hand, by the Lindenstrauss-Rosenthal theorem,
a separable space embeds into $\ell_\infty$ in a unique form \cite{lindrose}; thus we have isomorphisms
$$
\ell_\infty/\mathcal B\, \approx\, \ell_\infty/\mathcal C \oplus \ell_\infty/\mathcal D\, \approx\,  \ell_\infty/c_0 \oplus \ell_\infty/\mathcal D$$
and since $\ell_\infty/c_0 $ is not injective (a result by Amir \cite[Theorem 1.25]{2132}) neither $\ell_\infty/\mathcal B$ is.
\hfill $_\square$
\end{proof}

We pass now to new, maybe unexpected, results. Palamodov asked in \cite[Problem 6]{palah}: \emph{Is $\Ext^2(\cdot, Y)=0$ for any Fr\'{e}chet space}?
A (negative) solution to Palamodov's problem was provided by Wengenroth in \cite[Question~ 6]{weng}. A more concrete one in the domain of Banach spaces appears in \cite{castrica}. The question of whether $\Ext^2(\ell_2)=0$ was posed in \cite{cck}, reiterated in  \cite{castrica} and has been recently solved in the negative in \cite[Theorem 4.6]{ext2}.
Actually, it is shown in  \cite[Corollary 5.1]{ext2} that $\Ext^2(X, Y)\neq 0$ if  $X$ and $Y$ are  Banach spaces containing $\ell_2^n$ uniformly complemented, for instance if they have nontrivial type $p > 1$. Let us record the following easy remark before continuing:

\begin{lemma}\label{lem:complement}
If $A$ and $B$ are complemented subspaces of $X$ and $Y$, respectively, and $\Ext^n(X, Y)=0$, then $\Ext^n(A, B)=0$.
\end{lemma}

\begin{proof}
Let $\imath: A\To X$ and $\jmath: B\To Y$ be the inclusions and let $P:X\To A$ and $Q:Y\To B$ be the corresponding projections.
Every $\mathscr F\in \Ext^n(A, B)$ can be written as $Q\jmath\mathscr F P\imath$, with
$\jmath\mathscr F P\in \Ext^n(X, Y)$.
\hfill $_\square$
\end{proof}

Recall that the continuum hypothesis ({\sf CH}) is the statement $\aleph_1=\frak c$, while {\sf ZFC} is the usual setting of set theory, with the axiom of choice and {\sf MA} stands for Martin's axiom.

\begin{proposition}\label{prop:positive}$\;$
\begin{enumerate}
\item There exist $\mathscr L_\infty$ spaces $X$ for which $\Ext^2(X)\neq 0$.
\item Under {\sf CH}, $\Ext^2(X, c_0)\neq 0$ if $X$ is one of the spaces $c_0(\aleph_1), \ell_\infty, \ell_\infty/c_0$.
\end{enumerate}
\end{proposition}
\begin{proof}
Part (1) follows the idea of the proof of Proposition~\ref{bour}, using the preceding lemma: since
$$
\xymatrixrowsep{0.5pc}
\xymatrix{0\ar[r] &  c_0 \ar[r] &  \ell_\infty  \ar[rr] \ar[dr]  && \ell_\infty(\frak c)\ar[r] & {\underbrace{\ell_\infty(\frak c)\big/(\ell_\infty/c_0)}_{c\kappa^2(c_0)} } \ar[r]& 0 \\
& &&\ell_\infty/c_0 \ar[ur]
}
$$
is nonzero we can take $X=c_0\oplus \big(\ell_\infty(\frak c)\big/(\ell_\infty/c_0)\big)$ and the preceding lemma applies.

Part (2) follows from \cite[Theorem 1]{accgm5} where it has been shown that, under {\sf CH}, $\Ext(X,\ell_\infty/c_0)\neq 0$ for these choices of $X$. Therefore $X$ can replace $c\kappa^2(c_0)$ in the preceding diagram.
\hfill $_\square$
\end{proof}

The just proved result contains a difficult point inside: Is $\Ext^2(c_0(\aleph_1), c_0)=0$ in {\sf ZFC}? On one hand, $\Ext(c_0(\aleph_1), c_0)\neq 0$ in {\sf ZFC}
as it is witnessed by the well-known nontrivial exact sequence
$0\To c_0 \To C(\Delta_M)\To c_0(\aleph_1)\To 0$ in which
$C(\Delta_M)$ is the subspace of $\ell_\infty$ generated by $c_0$ and the characteristic functions of an almost disjoint family of size $\aleph_1$; see \cite[Example 2]{johnlind} or \cite[Section 2.2.4]{2132}.\\
On the other hand,  $\Ext^2(C(\Delta_M), c_0)=\Ext^2(c_0(\aleph_1), c_0)$ since
$\Ext(c_0)=\Ext^2(c_0)=0$. Finally, under [{\sf MA} + $\aleph_1<\mathfrak c$] one has $\Ext(C(\Delta_M), c_0)=0$ \cite[Corollary 5.3]{mp}, which opens the door to believe that also $\Ext^2(c_0(\aleph_1), c_0)=\Ext^2(C(\Delta_M), c_0)=0$ in this axiomatic.
\medskip

The situation for $\mathscr L_1$-spaces is completely different, as the following example of Wodzicki shows. The key point is that if
 $X$ is a separable $\mathscr L_1$-space not isomorphic to $\ell_1$, for instance  $X=L_1$, then $\kappa(X)$, which can assumed separable, is an $\mathscr L_1$-space not isomorphic to $\ell_1$. Actually $\kappa(X)$
is uncomplemented in its bidual: otherwise the projective presentation of $X$ would split (Theorem \ref{negative}(4)) forcing $X$ to be projective and thus isomorphic to $\ell_1$. Iterating the argument we obtain that the kernels $\kappa^n(X)$ are all $\mathscr L_1$-spaces not isomorphic to $\ell_1$ and that the sequences
$$
\xymatrixcolsep{2.00pc}
\xymatrixrowsep{0.5pc}
\xymatrix{
0 \ar[r] & \kappa^{n}X  \ar[r]   &  \ell_1 \ar[dr]   \ar[rr]  & & \cdots  \ar[dr] \ar[rr]  & & \ell_1 \ar[r]      &  X \ar[r] & 0\\
& & &\kappa^{n-1}X \ar[ur] && \kappa^{1}X \ar[ur] &
}
$$
are nonzero in $\Ext^n(X, \kappa^{n}X)$. So, $Y=X\oplus \kappa^{n}X$ is an $\mathscr L_1$-space for which $\Ext^n(Y)\neq 0$.

\medskip

We close this section with the following remark on $\Ext^3$. It is shown in \cite[Corollary  5.1]{ext2} that $\Ext^2(\ell_p)\neq 0$ for $1<p<\infty$ and it is a classical result in Banach space theory that $L_1$ contains isometric copies of $\ell_p$ for $1<p\leq 2$; see \cite[Theorem 6.4.17]{ak}.
These copies are uncomplemented, and so we have nontrivial sequences
$$
\xymatrix{
0\ar[r] & \ell_p\ar[r]^\imath &L_1\ar[r] & L_1/\imath[\ell_p] \ar[r] & 0 &(\mathscr F)
}
$$
Since $\Ext^2(L_1,\ell_p)=0$, by Theorem~\ref{negative}(4), taking any nonzero $\mathscr E\in \Ext^2(\ell_p)$ we have that $\mathscr{EF}$ is nonzero in $\Ext^3(L_1/\imath[\ell_p],\ell_p)$, by Lemma~\ref{lem:EF=0=>F=0}(b). Of course one also has $\Ext^2(L_1/\imath[\ell_p],\ell_p)\neq 0$: just consider $\mathscr{DF}$, with $\mathscr D$ nonzero in $\Ext(\ell_p)$.

\section{Homological dimension of Banach spaces}\label{sec:dim}

The study of the various homological dimensions of modules and algebras is a classical topic in the homology of Banach and topological algebras \cite[Chapter 7]{maclane}, \cite[III.6]{gelfman}, \cite[III.5]{helm}. In Banach spaces, however, the problem has only been considered, to the best of our knowledge, by Wodzicki \cite{wod}. Following \cite{wod}, we define
the  projective dimension $\pd(X)$ of a Banach space $X$ as the smallest $n$ for which $\Ext^{n+1}(X,\cdot)=0$ or, equivalently, the smallest $n$ so that $\kappa^n(X)$ is projective; analogously,
 the injective dimension ${\id}(X)$ is is the smallest $n$ for which $\Ext^{n+1}(\cdot, X)=0$ or $c\kappa^n(X)$ is injective.

 Wodzicki considers other variations such as the absolutely pure and pure injective dimensions and the \emph{flat dimension} of $X$, denoted $\fd(X)$, defined as the least integer $n$ for which there is an $n$-exact sequence
 $$
 \xymatrix{
 0 \ar[r] &\mathcal F_0  \ar[r] &\mathcal F_1\ar[r] &\dots \ar[r] &\mathcal F_n \ar[r] & X\ar[r] &0
 }
 $$
in which $\mathcal F_i$ are $\mathscr L_1$-spaces for all $0\leq i\leq n$.
This can be understood as a flat resolution of $X$ because the dual of an $\mathscr L_1$-space is already injective.

It is shown in \cite{wod} that $\pd(\mathcal B^*)=\id(\mathcal B^{**})$ (see Theorem~\ref{bour}) and also that $\pd(X)=\infty$ if $X$ is an $\mathscr L_1$-space not isomorphic to any $\ell_1(I)$; see the remarks closing the preceding section. This is essentially everything that is currently known about the behaviour of $\pd, \id, \fd$.
As remarked in \cite{wod}, it is expected these dimensions to be $\infty$ for most ``classical''
spaces, with the obvious exceptions.
However we do not have much evidence supporting this conjecture: actually we do not known how to construct large sequences with reflexive ends.
The obvious candidates to appear as ends are the following spaces, taken from \cite{jz}: Let $(G_n)_{n\geq 1}$ be  a sequence of finite dimensional spaces which is dense in the set of ``all finite dimensional spaces'' with respect to the Banach-Mazur distance in the sense that for every finite dimensional space $F$ and $\e>0$ there is some $n$ such that $d(F,G_n)<1+\e$. Define
$$
\mathcal C_ p=\begin{cases}
\ell_p(\mathbb N, G_n) &\text{for $1\leq p<\infty$},\\
c_0(\mathbb N, G_n) &\text{if $p=\infty$}.
\end{cases}
$$
These spaces test when a Banach space is an $\mathscr L_1$-space. Indeed
 $X$ is an $\mathscr L_1$-space if and only if $\Ext(X,\mathcal C_p)=0$ for some (equivalently, for every) $1 \leq  p < \infty$. If, besides, $X$ is separable, then $X$ is an $\mathscr L_1$-space if and only if $\Ext(X,\mathcal C_\infty)=0$.
A proof can be seen in \cite[Corollary 5.4]{castmoresob}.

The immediate consequence is that $\fd(X)$ is the least integer $n$ for which $0=\Ext^{n+1}(X,\mathcal C_p)$ for some (or any) $1\leq  p < \infty$. It is easy to believe that $\Ext^n(\mathcal C_p)\neq 0$ for all $n$, as it is the case for $n=1,2,3$ (we omit the proof). We have the following complement to \cite{wod} concerning Kadec space $\mathcal K$. This space, independently discovered by Kadec, Pe{\l}czy\'nski and Wojtaszczyk \cite{kade, pelcuni, p-w}, is separable, has the BAP, and it contains a complemented copy of each separable Banach space with the BAP.

\begin{proposition}
 $\Ext^n(\mathcal K)\neq 0$ for all $n$.
In particular $\pd\mathcal K= \id\mathcal K=\infty$.
\end{proposition}

\begin{proof}
As $\pd L_1=\infty$ for each $n\geq 1$ one has $\Ext^n(L_1,\kappa^n L_1)\neq 0$. Both $L_1$ and $\kappa^n L_1$ have the BAP (they are $\mathscr L_1$-spaces), they embed as complemented subspaces of $\mathcal K$ and so $\Ext^n(\mathcal K)\neq 0$, by Lemma~\ref{lem:complement}.
\hfill $_\square$
\end{proof}

Actually one can prove that if $X$ and $Y$ are separable Banach spaces, not necessarily having the BAP, such that $\Ext^n(X,Y)\neq 0$, then  $\Ext^n(X,\mathcal K)\neq 0$  and $\Ext^n(\mathcal K,Y)\neq 0$.
The following problem may be very hard, as only $\pd(\ell_2), \id(\ell_2)\geq 3$ is currently known:
\begin{prob}
 Compute the projective (or flat) and injective dimensions of the separable Hilbert space.
\end{prob}
\section{Banach vs. Quasi Banach spaces}\label{sec:Q}

Every Banach space is also a quasi Banach space and, therefore, each exact sequence of Banach spaces can be regarded as an exact sequence of quasi Banach spaces. (General references for quasi Banach spaces are the monograph \cite{KPR} and \cite{k-handbook}.)
It is then  natural to consider the interaction between the category {\bf Q} of quasi Banach spaces and its subcategory {\bf B} of Banach spaces, so let us add some remarks on this issue.
All the definitions and results in Section~\ref{sec:def-extn} and those in the appendix work in {\bf Q} exactly as in {\bf B}. In contrast, with the obvious exception of Lemma~\ref{lem:complement}, none of the results in Sections~\ref{sec:def-extn}, \ref{sec:red} and \ref{sec:dim} would survive in {\bf Q} since this category has no injective objects apart form 0 (this follows from \cite[Proof of Proposition 3.45]{2132}) and the only projective spaces are the finite dimensional ones: Indeed,
let $X$ be  a quasi Banach space. Then, by the  Aoki-Rolewicz Theorem (see \cite[Theorem 1.3]{KPR}) there is an index set $I$ and a quotient map $Q_p:\ell_p(I)\To X$ for suitable $0<p\leq 1$ and so for each $0<q<p$. If $X$ were projective in {\bf Q}, it would be isomorphic to a complemented subspace of $\ell_p(I)$ and to a complemented subspace of $\ell_q(I)$. It follows from a result of Stiles \cite[Theorem 2]{stiles} that $X$ is finite dimensional.

That said, let us write $\Ext^n_{\bf Q}$ to indicate exact sequences of quasi Banach spaces and $\Ext^n_{\bf B}$ when referring to the category of Banach spaces.
In spite of the fact that ${\bf Q}$ does not have enough injectives or projectives, it follows from the results in \cite{smirnov} that $\Ext^n_{\bf Q}(X,Y)$ are \emph{sets} when $X$ and $Y$ are quasi Banach spaces.

 The core problem is that it is perfectly possible to have two Banach spaces $X,Y$
and a short exact sequence
$$
\xymatrix{(\mathscr Z)&
0 \ar[r] & Y \ar[r]& Z \ar[r]  & X\ar[r] & 0 & &
}
$$
in which $Z$ is a quasi Banach space not isomorphic to a Banach space. That is,
$\Ext_{\bf Q}(X, Y)$ can be strictly larger than $\Ext_{\bf B}(X, Y)$. Perhaps the most extreme counterexample is obtained when $Y=\K$ is the ground field (which is injective in {\bf B} by the Hahn-Banach theorem) and $X=\ell_1$ (which is projective in {\bf B}), so in particular $\Ext_{\bf B}(\ell_1, \K)=0$.  However, Ribe \cite{ribe}, Kalton \cite{kalt} and Roberts \cite{robe}, independently and almost simultaneously around 1980, and Smirnov and Sheikhman \cite{smirsheik} around 1990, constructed examples of nontrivial elements of $\Ext_{\bf Q}(\ell_1, \K)$. Ribe's counterexample  is the simplest of the four and can be seen also in \cite[Chapter 5, \S~4]{KPR} and \cite[Section~4]{k-handbook}

There are also couples of Banach spaces for which $\Ext_{\bf Q}(X, Y)=\Ext_{\bf B}(X, Y)$, that is, any quasi Banach space $Z$ fitting in a short exact sequence as $(\mathscr Z)$ is necessarily (isomorphic to) a Banach space. Actually this depends only on the quotient space $X$. Indeed, if we agree to say that a quasi Banach space $X$ is a $K$-space when $\Ext_{\bf Q}(X, \K)=0$ then a classical result of Dierolf \cite{dier} shows that a Banach space $X$ is a $K$-space if and only if $\Ext_{\bf Q}(X,Y)=\Ext_{\bf B}(X,Y)$ for all Banach spaces $Y$.

While $\ell_1$ fails to be a $K$-space, other important families of Banach spaces are $K$-spaces, among them $B$-convex spaces \cite[Theorem 5.18]{KPR} as well as $\mathscr L_\infty$-spaces and their quotients \cite[Theorem 6.5]{kaltrobe}. This has the following consequence, where
$C[0,1]/\ell_1$ denotes any quotient of $C[0,1]$ by a subspace isomorphic to $\ell_1$.

\begin{proposition}\label{ejemplo} $\Ext_{\mathbf Q}^2\big(C[0,1]/\ell_1, \K\big)\neq 0$, while $\Ext_{\mathbf B}^2\big(C[0,1]/\ell_1, \K\big)=0$.
\end{proposition}
\begin{proof} The ``while''   part is clear since $\K$ is injective as a  Banach space.
To see the first part we apply the contravariant sequence (\ref{eq:conLES}) to
$$
\xymatrixcolsep{3pc}
\xymatrix{
(\mathscr C) & 0 \ar[r]  & \ell_1 \ar[r]^-\imath & C[0,1] \ar[r]^-\pi & C[0,1]/\imath[\ell_1] \ar[r] & 0 &
}
$$
with $B=\K$ and we look at
$$\xymatrixcolsep{1.5pc}
\xymatrix{
\dots \ar[r]  & \Ext_{\bf Q}(C[0,1], \K)  \ar[r]^-0 &
\Ext_{\bf Q}(\ell_1, \K)  \ar[rr]_-{\text{injective}}^-{\mathscr C^*} & & \Ext_{\bf Q}^2(C[0,1]/\imath[\ell_1], \K) \ar[r] &\dots
}
$$
The space $\Ext_{\bf Q}(C[0,1], \K)$ is zero, by the Kalton-Roberts theorem already mentioned; thus if we splice a nontrivial sequence in $\Ext_{\bf Q}(\ell_1, \K) $, for instance Ribe's
$$
\xymatrixcolsep{3.7pc}
\xymatrix{
(\mathscr R) & 0 \ar[r]  & \K\ar[r]& \mathcal R \ar[r] & \ell_1 \ar[r] & 0 &
}
$$
to $(\mathscr C)$ we get a nontrivial 2-exact sequence
$$
\xymatrixrowsep{0.5pc}\xymatrixcolsep{2.9pc}
\xymatrix{0\ar[r] &  \K \ar[r] & \mathcal R  \ar[rr] \ar[dr]  && C[0,1] \ar[r] & C[0,1]/\imath[\ell_1] \ar[r]& 0 \\
& &&\ell_1\ar[ur]
}
%\qedhere
$$
\hfill $_\square$
\end{proof}

\section{Appendix. The homology sequences}

\subsection{Pullback and pushout}
Given operators $\alpha:Y\To A$ and $\beta:Y\To B$ acting between Banach spaces, the associated pushout diagram is
\begin{equation}\label{po-dia1}
\begin{CD}
Y@>\alpha>> A\\
@V \beta VV @VV \overline \beta V\\
B @>> \overline \alpha > \PO
\end{CD}
\end{equation}
The pushout space $\PO=\PO(\alpha,\beta)$ in the quotient of the direct sum
$A\oplus_1 B$ by the closure of the subspace $\Delta=\{(\alpha y,-\beta y): y\in Y\}$.
The map $\overline \alpha$ is the composition of the inclusion of $B$ into $A\oplus_1 B$ and the natural
quotient map $A\oplus_1 B\to (A\oplus_1 B)/\overline\Delta$, so that
$\overline \alpha(b)=(0,b)+\overline\Delta$ and, analogously, $\overline \beta(a)=(a,0)+\overline\Delta$. All this make (\ref{po-dia1}) a commutative diagram: $\overline \beta\alpha=\overline \alpha\beta$. The pushout square (\ref{po-dia1}) has the following universal property: if $\beta':A\To C$ and  $\alpha':B\To C$ are operators such that $\beta'\alpha=\alpha'\beta$, there is a unique operator $\gamma:\PO\To C$ such that $\beta'=\gamma \overline \beta,  \alpha'=\gamma \overline \alpha$.

The pullback construction is the dual of that of pushout in the sense of categories,
that is, ``reversing arrows''.
Given operators $\alpha:A\To X$ and $\beta:B\To X$, the associated pullback diagram is
\begin{equation}\label{pb-dia}
\begin{CD}
B @> \beta >> X\\
@A{\underline \alpha}AA @AA \alpha A\\
\PB@>>\underline \beta > A
\end{CD}
\end{equation}
The pullback space is $\PB=\PB(\alpha,\beta)=\{(b,a)\in B\oplus_\infty A: \beta (b)=\alpha(a) \}$.
The underlined arrows are the restriction of the projections onto the corresponding factor. The pullback square has the following universal property: if $\alpha':C\To B, \beta':C\To A$ are operators such that $\beta\alpha'=\alpha\beta'$, then there exists a unique operator $\gamma: C\To \PB$ satisfying $ {\underline \alpha} \gamma=\alpha'$ and
$ {\underline \beta} \gamma=\beta'$.
Given an exact sequence
$$
\xymatrixcolsep{2.7pc}
\xymatrix{(\mathscr F)&0\ar[r] & Y\ar[r]^{f_0} &F_1\ar[r]^{f_1} &\cdots\ar[r]  &F_n \ar[r]^{f_n}  &X\ar[r] & 0
}
$$
and an  operator $\beta: Y\to B$, the pushout sequence $ \beta\mathscr F$ is the lower sequence in the diagram
$$
\xymatrixcolsep{2.15pc}
\xymatrix{(\mathscr F) & 0\ar[r] & Y\ar[r]^{f_0}\ar[d]_\beta &F_1\ar[r]^{f_1}\ar[d]^{\overline{\beta}} &F_2\ar[r]^{f_2} \ar@{=}[d]  &\cdots\ar[r]  &F_n \ar[r]^{f_n} \ar@{=}[d]   &X\ar[r] \ar@{=}[d]  & 0 \\
(\beta\mathscr F)& 0\ar[r] & B\ar[r]^{\overline{f_0}} &\PO \ar[r]^{\overline{f_1}} &F_2\ar[r]^{f_2}  &\cdots\ar[r]  &F_n \ar[r]^{f_n}  &X\ar[r] & 0
}
$$
Here, the left square is the pushout of the operators $f_0$ and $\beta$, while $\overline{f_1}$ is obtained from $f_1$ and the null map $0:B\To F_2$ and the universal property of $\PO$.

Dually, if $\alpha: A\To X$ is an operator, the pullback of $\mathscr F$ and $\alpha$ is the lower sequence in the commutative diagram
$$
\xymatrixcolsep{2.1pc}
\xymatrix{  0\ar[r] & Y\ar[r]^{f_0}\ar@{=}[d] &F_1 \ar@{=}[d]\ar[r]&\dots\ar[r]  &F_{n-1}\ar[r]^{f_{n-1}} \ar@{=}[d]   &F_n \ar[r]^{f_n}    &X\ar[r]  & 0 &(\mathscr F) \\
0\ar[r] & Y\ar[r]^{f_0} &F_1 \ar[r]&\dots\ar[r]  &F_{n-1}\ar[r]^{\underline{f_{n-1}}}  &\PB \ar[r]^{\underline{f_n}} \ar[u]^{\underline{\alpha}}  &A\ar[r]\ar[u]^\alpha   & 0 &(\mathscr F\!\alpha)
}
$$
The right square is the pullback of the operators $f_n$ and $\alpha$ and $\underline{f_{n-1}}$ is obtained from ${f_{n-1}}$ and $0:F_{n-1}\To A$ by the universal property of $\PB$.

It is clear that if $\mathscr E\sim \mathscr F$, then $\beta \mathscr E\sim \beta \mathscr F$ and $\mathscr E\alpha \sim \mathscr F\alpha$.

\subsection{Two long sequences}
The long homology sequences (also known as the ``Hom-Ext sequences'') connect spaces of operators and the successive $\Ext^n$. The following description suffices to understand everything in this paper. We begin with the covariant case.
Let
$$
\xymatrix{(\mathscr Z)&
0 \ar[r] & Y \ar[r]^\imath & Z \ar[r]^\pi & X\ar[r] & 0 & &
}
$$
be a short exact sequence
 and let $A$ be another Banach space. Then the following sequence is exact:
\begin{equation}\label{eq:covLES}
\xymatrixcolsep{2.0pc}
\xymatrixrowsep{0.8pc}
\xymatrix{
0 \ar[r]&\frak L(A, Y) \ar[r]^{\imath_*} & \frak L(A,Z) \ar[r]^{\pi_*} & \frak L(A,X)\\
\ar[r]^-{\mathscr Z_*}& \Ext(A,Y) \ar[r]^{\imath_*}  &\Ext(A, Z) \ar[r]^{\pi_*} &\Ext(A,X) \\
\cdots\ar[r]^-{\mathscr Z_*}& \Ext^n(A,Y) \ar[r]^{\imath_*}  &\Ext^n(A, Z) \ar[r]^{\pi_*} &\Ext^n(A,X) \\
\ar[r]^-{\mathscr Z_*}& \Ext^{n+1}(A,Y)  \ar[r]^{\imath_*}  &\Ext^{n+1}(A, Z) \ar[r]^{\pi_*} &\Ext^{n+1}(A,X)\ar[r]&\dots \\
}
\end{equation}
We apologize for the plethora of labels. Let us explain the meaning of the arrows. The first ocurrence of $\imath_*$ and $\pi_*$ is simple composition on the left: if $a\in\frak L(A,Y)$, then $\imath_*(a)= \imath a$ and the same applies to $\pi_*$. The first $\mathscr Z_*$ takes an operator $a:A\To X$ into the pullback $\mathscr Z a$.
The remaining $\imath_*$ and $\pi_*$ act taking pushouts: if $\mathscr E\in\Ext^n(A,Y)$, then $\imath_*(\mathscr E)=\imath \mathscr E$ is the lower sequence in the pushout diagram
$$\xymatrixcolsep{2.7pc}
\xymatrix{
0\ar[r] & Y\ar[r]\ar[d]^{{\imath}} & E_n\ar[r] \ar[d]^{\overline{\imath}}  & E_{n-1}\ar[r] \ar@{=}[d]  & \dots \ar[r] &E_1\ar[r]\ar@{=}[d] & A\ar[r] \ar@{=}[d] &0\\
0\ar[r] & Z\ar[r] & \PO \ar[r]  & E_{n-1}\ar[r]  & \dots \ar[r] &E_1\ar[r] & A\ar[r] &0
}
$$
The same applies to $\pi_*$. The remaining $\mathscr Z_*$ act by splicing through $X$: if  $\mathscr F\in\Ext^n(A,X)$, then $\mathscr Z_*(\mathscr F)= \mathscr Z \mathscr F$, as in the diagram
$$
\xymatrixrowsep{0.5pc}\xymatrixcolsep{2.4pc}
\xymatrix{
0\ar[r] & Y\ar[r]  & Z \ar[rd]\ar[rr]  &  & F_{1} \ar[r]  &\dots \ar[r] &F_n \ar[r] & A \ar[r] & 0\\
&& & X \ar[ur]
}
$$
This concludes the description of (\ref{eq:covLES}). For a proof of its linearity and exactness, see, for instance \cite[Theorem 6.42]{fresi}---or \cite[VII. Theorem 5.1]{mitc} if you want to learn the original proof by Schanuel.

We pass to describe, even more succinctly, the contravariant sequence. We consider again $(\mathscr Z)$ and a new ``target'' space $B$. Then the following sequence is exact
\begin{equation}\label{eq:conLES}
\xymatrixcolsep{2.5pc}
\xymatrixrowsep{0.8pc}
\xymatrix{
0 \ar[r]&\frak L(X,B) \ar[r]^{\pi^*} & \frak L(Z,B) \ar[r]^{\imath^*} & \frak L(Y, B)\\
\ar[r]^-{\mathscr Z^*}& \Ext(X,B) \ar[r]^{\pi^*}  &\dots \ar[r]^{\imath^*} &\Ext^{n-1}(X,B) \\
\ar[r]^-{\mathscr Z^*}& \Ext^n(X,B) \ar[r]^{\pi^*}  &\Ext^n(Z,B) \ar[r]^{\imath^*} &\Ext^n(X,B)\ar[r]^-{\mathscr Z^*}&\dots
}
\end{equation}
The meaning of the arrows should be obvious: the first occurrences of $\pi^*$ and $\imath^*$ act by composition on the right; all other by forming pullbacks. As for the arrows labelled as $\mathscr Z^*$ the first one acts forming pushouts and the remaining ones by splicing. The exactness of the sequence is proved in \cite[Theorem 6.43]{fresi}.

%\end{fulltext}

\end{document}